\documentclass[12pt,a4paper]{amsart}
\usepackage{amsfonts,amsmath,amsthm,amssymb}
\usepackage{graphics}
\usepackage{pstricks,pst-node}
\usepackage[headings]{fullpage}
\usepackage{rotating}

\newtheorem{theorem}{Theorem}[section]
\newtheorem{proposition}[theorem]{Proposition} 
\newtheorem{lemma}[theorem]{Lemma}             
\newtheorem{corollary}[theorem]{Corollary}       

\theoremstyle{definition}
\newtheorem{definition}[theorem]{Definition}   

\theoremstyle{remark}
\newtheorem{example}[theorem]{Example}         

\newtheorem*{examplesn}{Example~\ref{exStan} (continued)}
\newtheorem*{examplef}{Example~\ref{exStan} (conclusion)}

\DeclareMathOperator\inv{inv}
\DeclareMathOperator\maxinv{maxinv}

\newcommand\R{\mathbb{R}}
\newcommand\N{\mathbb{N}}

\newcommand\fS{\mathfrak{S}}
\newcommand\fI{\mathfrak{I}}

\newcommand\cS{\mathcal{S}}

\newcommand\tw{\tilde{w}}
\newcommand\tf{\tilde{f}}

\newcommand\tI{\tilde{I}}
\newcommand\tfI{\tilde{\mathfrak{I}}}

\newcommand\PFA{\operatorname{PF}\mkern-8mu{}_A}
\newcommand\CFA{\operatorname{CF}\mkern-8mu{}_A}
\newcommand\WA{\operatorname{W}\mkern-8mu{}_A}
\newcommand\phiA{\varphi\mkern-4mu{}_A}

\newcommand\WZ{\operatorname{W}\mkern-6mu{}_Z}
\newcommand\phiZ{\varphi\mkern-2mu{}_Z}
\newcommand\WtZ{\operatorname{W}\mkern-6mu{}_{\tilde{Z}}}
\newcommand\phitZ{\varphi\mkern-2mu{}_{\tilde{Z}}}
\DeclareMathOperator\ct{C_A}

\newcommand\tvs{\strut\negthickspace\text{\Large\textvisiblespace}\negthickspace}

\newcommand\pequr[2]{$\llap{\scriptsize$#1-#2=1$}$}

\newcommand\peqzr[2]{$\llap{\scriptsize$#1-#2=0$}$}
\newcommand\sr[2]{\begin{rotate}{#2}#1\end{rotate}}

\newcommand\ob[1]{\widehat{#1}}
\newcommand\w[2]{w\langle#1\!:\!#2\rangle}

\title{Note on the bijectivity of the Pak-Stanley labelling}

\author[R.~Duarte]{Rui Duarte}
\address{Center for Research and Development in Mathematics and Applications,
Department of Mathematics, University of Aveiro}
\email{rduarte@ua.pt}

\author[A.~Guedes~de~Oliveira]{Ant\'onio Guedes
  de Oliveira}
\address{CMUP and Department of Mathematics, Faculty of Sciences,
University of Porto}
\email{agoliv@fc.up.pt}

\thanks{The work of both authors was supported in part by the European
  Regional Development Fund through the program COMPETE - Operational
  Program Factors of Competitiveness (``Programa Operacional Factores
  de Competitividade'') and by the Portuguese government through FCT -
  Funda\c{c}\~ao para a Ci\^encia e a Tecnologia, under the projects
  PEst-C/MAT/UI0144/2014 and PEst-C/MAT/UI4106/2014.}

\begin{document}

\maketitle

\section{Introduction}
\noindent
This article has the sole purpose of presenting a simple,
self-contained and direct proof of the fact that the Pak-Stanley
labeling is a bijection.  The construction behind the proof is
subsumed in a forthcoming paper \cite{DGO}, but an actual
self-contained proof is not explicitly included in that paper.

Let $n$ be a natural number and consider the \emph{Shi arrangement of
  order $n$}, the union $\cS_n$ of the hyperplanes of $\R^n$ defined,
for every $1\leq i<j\leq n$, either by equation $x_i-x_j=0$ or by
equation $x_i-x_j=1$. The \emph{regions} of the arrangement are the
connected components of the complement of $\cS_n$ in
$\R^n$. Jian~Yi~Shi \cite{Shi} introduced in literature this
arrangement of hyperplanes and showed that the number of regions is
$(n+1)^{n-1}$.

On the other hand, $(n+1)^{n-1}$ is also the number of \emph{parking
  functions of size $n$}, which were defined (and counted) by
Alan~Konheim and Benjamin~Weiss \cite{KW}.  These are the functions
$f\colon[n]\to[n]$ such that,
$$\forall_{i\in[n]}\,,\ |f^{-1}([i])|\geq i$$
or, equivalently, such that, for some $\pi\in\fS_n$, $f(i)\leq \pi(i)$
for every $i\in[n]$ (as usual, $[n]:=\{1,\dotsc,n\}$ and $\fS_n$ is
the set of permutations of $[n]$).

The Pak-Stanley labeling \cite{stan2} consists of a function $\lambda$
from the set of regions of $\cS_n$ to the set of parking functions of
size $n$.

We define $[0]:=\varnothing$ and, for $i,j\in\N$,
$[i,j]:=[j]\setminus[i-1]$, so that $[i,j]=\{i,i+1,\dotsc,j\}$ if
$i\leq j$ and $[i,j]=\varnothing$ otherwise. Finally, $[i]=[1,i]$ for
every integer $i\geq 0$ as stated before.

Let $A\subseteq[n]$, say $A=:\{a_1,\dotsc,a_m\}$ with $a_1<\dotsb<a_m$
and let $\WA$ be the set of words of form $w=a_{\alpha_1}\dotsb
a_{\alpha_m}$ for some permutation $\alpha\in\fS_m$. If $1\leq i<j\leq
m$, we distinguish the subword $\w{i}{j}:=a_{\alpha_i}\dotsb
a_{\alpha_j}$ from the set $w([i,j]):=\{a_{\alpha_i},\dotsc,
a_{\alpha_j}\}$. Similarly, we define $w^{-1}\colon A\to[m]$ through
$w^{-1}(w_i)=i$ for every $i\in[m]$.

\begin{definition}\label{contract}
  Given a word $w=w_1\dotsb w_k\in\WA$ and a set
  $\fI=\{[o_1,c_1],\dotsc,[o_k,c_k]\}$ with $1\leq o_i<c_i\leq m$ for
  every $i\in[k]$ and $o_1<o_2<\dotsb<o_k$, we say that the pair
  $P=(w,\fI)$ is a \emph{valid pair} if
\begin{itemize}
\item $w_{o_i}>w_{c_i}$ for every $i\in[k]$;
\item $c_1<c_2<\dotsb<c_k$.
\end{itemize}
An \emph{$A$-parking function} is a function $f\colon A\to[m]$ for
which
\begin{equation}\label{central}
\forall_{j\in[m]}\,,\ |f^{-1}([j])|\geq j\,.
\end{equation}
We denote by $\PFA$ the set of $A$-parking functions.  Of course, for
$f\colon A\to[m]$, $f\in\PFA$ if and only if $f\circ\iota_A$ is a
parking function, where $\iota_A\colon[m]\to A$ is such that
$\iota_A(i)=a_i$.  A particular case occurs when
$$\forall_{j\in[m]}\,,\ f(a_j)\leq j\,.$$
In this case, we say that $f$ is \emph{$A$-central}. We denote by
$\CFA$ the set of $A$-central parking functions.  We call
\emph{contraction of $w$} to the new function $\ob{w}\colon A\to[m]$
such that
\begin{align}
  &\ob{w}(a):=w^{-1}(a)-\Big|\big\{b\in A\mid b>a\,,\
  w^{-1}(b)<w^{-1}(a) \big\}\Big|\,.
\end{align}
Note that indeed $\ob{w}\in\CFA$, since
$\ob{w}(a)=\Big|w([w^{-1}(a)])\cap[a]\Big|$.
\end{definition}

For example, $\ob{843967}=
\overset{{\scriptscriptstyle3}}{1}\overset{{\scriptscriptstyle4}}{1}%
\overset{{\scriptscriptstyle6}}{3}\overset{{\scriptscriptstyle7}}{4}%
\overset{{\scriptscriptstyle8}}{1}\overset{{\scriptscriptstyle9}}{4}$.
In fact, $\ob{843967}(3)=1$ since $w^{-1}(3)=3$ and
$w([3])\cap[3]=\{8,4,3\}\cap[3]=\{3\}$, but, for instance,
$\ob{843967}(6)=3$ since $w^{-1}(6)=5$ and $w([5])\cap[6]=\{3,4,6\}$.

When $A=[n]$, the $A$-central parking functions are simply
\emph{central} parking functions.

\section{The Pak-Stanley labeling}

Igor Pak and Richard Stanley \cite{stan2} created a (bijective)
labeling of the regions of the Shi arrangement with parking functions
that may be defined as follows.

Consider, for a point $x=(x_1,\dotsc,x_n)\in\R^n\setminus\cS_n$, the
(unique) permutation $w\in\fS_n$ such that
$x_{w_1}<\dotsb<x_{w_n}$\footnote{Note that the order is reversed
  relatively to Stanley's paper \cite{stan2}.}, and consider the set
$\fI=\big\{[o_1,c_1],\dotsc,[o_m,c_m]\big\}$ of all \emph{maximal}
intervals $I_i=[o_i,c_i]$ with $o_i<c_i$ for $i=1,\dotsc,k$, such that
\begin{itemize}
\item $w_{o_i}>w_{c_i}$;
\item for every $\ell,m\in I_i$ with $\ell<m$ and $w_\ell>w_m$,
  $0<x_{w_m}-x_{w_\ell}<1$\footnote{The fact that
    $0<x_{w_m}-x_{w_\ell}$ already follows from the fact that
    $w_\ell>w_m$.}.

\end{itemize}

Then, clearly $(w,\fI)$ is a valid pair that does not depend on the
particular point $x$ that we have chosen. More precisely, if a similar
construction is based on a different point $y\in\R^n\setminus\cS_n$
then at the end we obtain the same valid pair if and only if $x$ and
$y$ are in the same region of $\cS_n$.  Finally, it is not difficult
to see that every valid pair corresponds in this way to a (unique)
region of $\cS_n$.
\begin{example}[{\cite[example p.\ 484, ad.]{stan1}}]\label{exStan}
  Let $w=843967125$ and $\fI=\{[1,6],[3,8],[6,9]\}$. The valid pair
  $(w,\fI)$ corresponds to the region
\begin{align*}
  \Big\{(x_1,\dotsc,x_9)\in\mathbb{R}^9\;\mid\;
  &x_{8}<x_4<x_{3}<x_9<x_6<x_{7}< x_1<x_{2}<x_{5},\\
  & x_8+1>x_7,\,x_3+1>x_2,\, x_7+1>x_5\,,\\[-3pt]
  &x_4+1<x_1,\, 
  x_6+1<x_5\Big\}
\end{align*}
where also $x_8+1>x_6$ (since $x_7>x_6$) and $x_8+1<x_1$ (since
$x_8<x_4$), for example.
\end{example}

Let $R_0$ be the region corresponding to the valid pair $(w,\fI)$
where $w=n(n-1)\dotsb\,2\,1$ and $\fI=\{[1,n]\}$, so that
$(x_1,\dotsc,x_n)\in R_0$ if and only if $0<x_i-x_j<1$ for every
$0\leq i<j\leq n$.

In the Pak-Stanley labeling $\lambda$, the label of $R_0$ is, using
the one-line notation, $\lambda(R_0)=1\,1\dotsb1$.  Furthermore,
\begin{itemize}
\item if the only hyperplane that separates two regions, $R$ and $R'$, has
  equation $x_i=x_j$ ($i<j$) 
  and $R_0$ and $R$ lie in the same side of this plane,
  then $\lambda(R')=\lambda(R)+e_j$ (as usual, the $i$-th coordinate
  of $e_j$ is either $1$, if $i=j$, or $0$, otherwise);
\item if the only hyperplane that separates two regions, $R$ and $R'$, has
  equation $x_i=x_j+1$ ($i<j$) and $R_0$ and $R$ lie in the same side
  of this plane, then $\lambda(R')=\lambda(R)+e_i$.
\end{itemize}

Thus, given a region $R$ of $\cS_n$ with associated valid pair
$P=(w,\{[o_1,c_1],\dotsc,[o_m,c_m]\})$, if $f=\lambda(R)$ and
{\boldmath $i=w_j$}, then, counting the planes of equation
$x_{w_k}-x_i=0$ or $x_i-x_{w_k}=1$ that separate $R$ and $R_0$,
respectively, we obtain (cf. \cite{stan2})
\begin{align}
\begin{split}\label{pdef}
  &f_i=1+\Big|\big\{k<j\mid w_k<i\big\}\Big|\\
  &\hphantom{f_i=1}{}+\Big|\big\{k<j\mid w_k>i\,,\ \text{no $\ell\in[m]$
    satisfies $j,k\in [o_\ell,c_\ell]$}\big\}\Big|\,.
\end{split}
  \intertext{Hence, if $j\notin [o_1,c_1],\dotsc,[o_m,c_m]$,}
  &f_i=j\,;\label{defp1}\\
  \intertext{in this case, \emph{let $o_P(i)=o_P(w_j):=j$}. Otherwise, if $k\leq m$ is
    the least integer for which $j\in[o_k,c_k]$,} &f_i=
  o_k-1+\ob{\w{o_k}{c_k}}(i)\,.
  \label{defp2}
\end{align}
and we \emph{define $o_P(i):=o_k$}.

In Figure~\ref{fig1}, we represent $\cS_3$ with each region $R$
labeled with $\lambda(R)$.

By requiring the validity of equations \eqref{defp1} and \eqref{defp2}
under the same conditions, we extend $\lambda$ to every valid pair
$P=(w,\fI)$, where $w\in\WA$ for some $A\subseteq[n]$. Note that in
this way we still obtain an $A$-parking function $f=\lambda(w,\fI)$.

Moreover, if $1\leq k<\ell\leq|A|$ then $o_P(w_k)\leq o_P(w_\ell)$.
If, in addition, $w_k>w_\ell$, then
\begin{equation}\label{comp}f(w_k)\leq f(w_\ell)\,.\end{equation}
In fact, $f(w_\ell)=\ell-\big|\big\{o_P(w_\ell)\leq j\leq\ell\mid
w_j>w_\ell\big\} \big|\geq k-\big|\big\{o_P(w_k)\leq j\leq k\mid
w_j>w_k\big\}\big|=f(w_k)$, since the size of the set
$\big\{o_P(w_\ell)\leq j\leq\ell\mid w_j>w_\ell\big\}\setminus
\big\{o_P(w_k)\leq j\leq k\mid w_j>w_k\big\}$, which is equal to
$\big\{k<j\leq\ell\mid w_\ell<w_j\leq w_k\big\}$, is clearly less than
or equal to $\ell-k$.

\begin{examplesn}
  Let again $R$ be the region of $\cS_9$ associated with the valid
  pair $\big(843967125,\{[1,6],[3,8],[6,9]\}\big)$.  Writing with a
  variant of Cauchy's two-line notation, we have, corresponding to the
  intervals $[1,6]$, $[3,8]$ and $[6,9]$, respectively,
  $\w{1}{6}=843967$ and $f_1=\ob{843967}=
  \overset{{\scriptscriptstyle3}}{1}\overset{{\scriptscriptstyle4}}{1}%
  \overset{{\scriptscriptstyle6}}{3}\overset{{\scriptscriptstyle7}}{4}%
  \overset{{\scriptscriptstyle8}}{1}\overset{{\scriptscriptstyle9}}{4}$,
  $f_2=\ob{396712}=
  \overset{{\scriptscriptstyle1}}{1}\overset{{\scriptscriptstyle2}}{2}%
  \overset{{\scriptscriptstyle3}}{1}\overset{{\scriptscriptstyle6}}{2}%
  \overset{{\scriptscriptstyle7}}{3}\overset{{\scriptscriptstyle9}}{2}$,
  $f_3=\ob{7125}=
  \overset{{\scriptscriptstyle1}}{1}\overset{{\scriptscriptstyle2}}{1}%
  \overset{{\scriptscriptstyle5}}{3}\overset{{\scriptscriptstyle7}}{1}$
  and, finally, $f=\lambda(R)=341183414$, which we also write
  $\vphantom{\text{\large$\int_I^I$}}\rnode{c8}{8}4\rnode{c3}{3}96
  \rnode{c7}{7}1\rnode{c2}{2}\rnode{c5}{5}
  \nccurve[nodesep=.035,ncurv=0.35,angleB=90,angleA=90]{-}{c8}{c7}
  \nccurve[nodesep=.035,ncurv=0.35,angleB=90,angleA=90]{-}{c7}{c5}
  \nccurve[nodesep=.035,ncurv=0.35,angleB=90,angleA=90]{-}{c3}{c2}$
  \footnote{Note that, for example, the central parking function
    $1132=\ob{2413}$ corresponds to $\rnode{c2}{2}\rnode{c4}{4}\rnode{c1}{1}\rnode{c3}{3}$.
    \nccurve[nodesep=.055,ncurv=0.55,angleB=90,angleA=90]{-}{c2}{c1}%
    \nccurve[nodesep=.055,ncurv=0.55,angleB=90,angleA=90]{-}{c4}{c3}}
  (cf. Figure~\ref{fig1}).

  Similarly, for $A=[9]\setminus\{8,4\}$, we may consider
  $f=\lambda\big(3967125,\{[1,6],[4,7]\}\big)$, the $A$-parking
  function $\rnode{d3}{3}96\rnode{d7}{7}1\rnode{d2}{2}\rnode{d5}{5}
  \nccurve[nodesep=.035,ncurv=0.35,angleB=90,angleA=90]{-}{d7}{d5}
  \nccurve[nodesep=.035,ncurv=0.35,angleB=90,angleA=90]{-}{d3}{d2}=
  \overset{{\scriptscriptstyle1}}{1}\overset{{\scriptscriptstyle2}}{2}
  \overset{{\scriptscriptstyle3}}{1}\overset{{\scriptscriptstyle5}}{6}
  \overset{{\scriptscriptstyle6}}{2}\overset{{\scriptscriptstyle7}}{3}
  \overset{{\scriptscriptstyle9}}{2}$.
\end{examplesn}

\begin{figure}[th]
\begin{center}
\bigskip
\setlength{\unitlength}{.55cm}
\begin{picture}(14.,15.)(0.,-14.)
\psset{unit=.6cm}
\psset{dotsize=.1 .1}
\rput(13.,-5.25){\sr{\pequr{x}{y}}{-20}}
\rput(13.,-6.75){\sr{\peqzr{x}{y}}{17.5}}
\rput(9.567,-0.553){\sr{\pequr{x}{z}}{45}}
\rput(10.867,-1.304){\sr{\peqzr{x}{z}}{80}}
\rput(9.133,-11.696){\sr{\pequr{y}{z}}{-40}}
\rput(10.433,-10.946){\sr{\peqzr{y}{z}}{-75}}
\rput(14.5,-8){\sr{{\small\red$\infty$}}{75}}
\pscircle[linecolor=red](7,-6){7.}
\pscurve[linecolor=black,showpoints=true,linewidth=.07]{-}(3.5,0.02)(7,
-4)(8.74,-7)(10.5,-12.02)
\pscurve[linecolor=black,linewidth=.04]{-}(3.5,0.02)(5.26,-7)(10.5,-12.02)
\pscurve[linecolor=black,linewidth=.04]{-}(10.5,0.02)(7,-4)(5.26,-7)(3.5,
-12.02)
\pscurve[linecolor=black,showpoints=true,linewidth=.07]{-}(10.5,0.02)(8.74,
-7)(3.5,-12.02)
\pscurve[linecolor=black,linewidth=.04]{-}(0,-6)(7,-4)(14,-6)
\pscurve[linecolor=black,showpoints=true,linewidth=.07]{-}(0,-6)(5.26,
-7)(8.74,-7)(14,-6)
\rput(7,-1.){$\begin{array}{c}
\ \\[-5pt]
312\\
\rnode{a1}{\scriptstyle\red 2}\,
\rnode{a2}{\scriptstyle\red 3}\,
\rnode{a3}{\scriptstyle\red 1}
\end{array}$}
\rput(7,-5.75){$\begin{array}{c}
\ \\[-5pt]
111\\
\rnode{a1}{\scriptstyle\red 3}\,
\rnode{a2}{\scriptstyle\red 2}\,
\rnode{a3}{\scriptstyle\red 1}
\nccurve[linecolor=red,nodesep=.05,ncurv=0.75,angleB=90,angleA=90]{-}{a1}{a3}
\end{array}$}
\rput(7,-7.75){$\begin{array}{c}
\ \\[-7.5pt]
121\\
\rnode{a1}{\scriptstyle\red 3}\,
\rnode{a2}{\scriptstyle\red 1}\,
\rnode{a3}{\scriptstyle\red 2}
\nccurve[linecolor=red,nodesep=.05,ncurv=0.75,angleB=90,angleA=90]{-}{a1}{a3}
\end{array}$}
\rput(7,-11){$\begin{array}{c}
\ \\[-5pt]
132\\
\rnode{a1}{\scriptstyle\red 1}\,
\rnode{a2}{\scriptstyle\red 3}\,
\rnode{a3}{\scriptstyle\red 2}\\
\end{array}$}
\rput(9.5,-5.){\llap{$\begin{array}{c}
\ \\[-5pt]
112\\
\rnode{a1}{\scriptstyle\red 2}\,
\rnode{a2}{\scriptstyle\red 3}\,
\rnode{a3}{\scriptstyle\red 1}
\nccurve[linecolor=red,nodesep=.05,ncurv=0.75,angleB=90,angleA=90]{-}{a1}{a3}
\end{array}$}}
\rput(4.5,-5.){\rlap{$\begin{array}{c}
\ \\[-5pt]
211\\
\rnode{a1}{\scriptstyle\red 3}\,
\rnode{a2}{\scriptstyle\red 2}\,
\rnode{a3}{\scriptstyle\red 1}
\nccurve[linecolor=red,nodesep=.05,ncurv=0.75,angleB=90,angleA=90]{-}{a1}{a2}
\nccurve[linecolor=red,nodesep=.05,ncurv=0.75,angleB=90,angleA=90]{-}{a2}{a3}
\end{array}$}}
\rput(9.82,-5.5){\rlap{$\begin{array}{c}
\ \\[-5pt]
113\\
\rnode{a1}{\scriptstyle\red 2}\,
\rnode{a2}{\scriptstyle\red 1}\,
\rnode{a3}{\scriptstyle\red 3}
\nccurve[linecolor=red,nodesep=.05,ncurv=0.75,angleB=90,angleA=90]{-}{a1}{a2}
\end{array}$}}
\rput(4.18,-5.5){\llap{$\begin{array}{c}
\ \\[-5pt]
221\\
\rnode{a1}{\scriptstyle\red 3}\,
\rnode{a2}{\scriptstyle\red 2}\,
\rnode{a3}{\scriptstyle\red 1}
\nccurve[linecolor=red,nodesep=.05,ncurv=0.75,angleB=90,angleA=90]{-}{a2}{a3}
\end{array}$}}
\rput(11.8,-2.95){$\begin{array}{c}
\ \\[-5pt]
213\\
\rnode{a1}{\scriptstyle\red 2}\,
\rnode{a2}{\scriptstyle\red 1}\,
\rnode{a3}{\scriptstyle\red 3}\\
\end{array}$}
\rput(2.2,-2.95){$\begin{array}{c}
\ \\[-5pt]
321\\
\rnode{a1}{\scriptstyle\red 3}\,
\rnode{a2}{\scriptstyle\red 2}\,
\rnode{a3}{\scriptstyle\red 1}\\
\end{array}$}
\rput(11.35,-8.3){$\begin{array}{c}
\ \\[-5pt]
123\\
\rnode{a1}{\scriptstyle\red 1}\,
\rnode{a2}{\scriptstyle\red 2}\,
\rnode{a3}{\scriptstyle\red 3}\\
\end{array}$}
\rput(2.65,-8.3){$\begin{array}{c}
\ \\[-5pt]
231\\
\rnode{a1}{\scriptstyle\red 3}\,
\rnode{a2}{\scriptstyle\red 1}\,
\rnode{a3}{\scriptstyle\red 2}\\
\end{array}$}
\rput(9.5,-8.7){\llap{$\begin{array}{c}
\ \\[-5pt]
122\\
\rnode{a1}{\scriptstyle\red 1}\,
\rnode{a2}{\scriptstyle\red 3}\,
\rnode{a3}{\scriptstyle\red 2}
\nccurve[linecolor=red,nodesep=.05,ncurv=0.75,angleB=90,angleA=90]{-}{a2}{a3}
\end{array}$}}
\rput(4.5,-8.7){\rlap{$
\begin{array}{c}
\ \\[-5pt]
131\\
\rnode{a1}{\scriptstyle\red 3}\,
\rnode{a2}{\scriptstyle\red 1}\,
\rnode{a3}{\scriptstyle\red 2}
\nccurve[linecolor=red,nodesep=.05,ncurv=0.75,angleB=90,angleA=90]{-}{a1}{a2}
\end{array}$}}
\rput(9.87,-2.6){\llap{$\begin{array}{c}
\ \\[-5pt]
212\\
\rnode{a1}{\scriptstyle\red 2}\,
\rnode{a2}{\scriptstyle\red 3}\,
\rnode{a3}{\scriptstyle\red 1}
\nccurve[linecolor=red,nodesep=.05,ncurv=0.75,angleB=90,angleA=90]{-}{a2}{a3}
\end{array}$}}
\rput(4.13,-2.6){\rlap{$\begin{array}{c}
\ \\[-5pt]
311\\
\rnode{a1}{\scriptstyle\red 3}\,
\rnode{a2}{\scriptstyle\red 2}\,
\rnode{a3}{\scriptstyle\red 1}
\nccurve[linecolor=red,nodesep=.05,ncurv=0.75,angleB=90,angleA=90]{-}{a1}{a2}
\end{array}$}}
\end{picture}
\caption{Pak-Stanley labeling for $n=3$}
\label{fig1}
\end{center}
\end{figure}
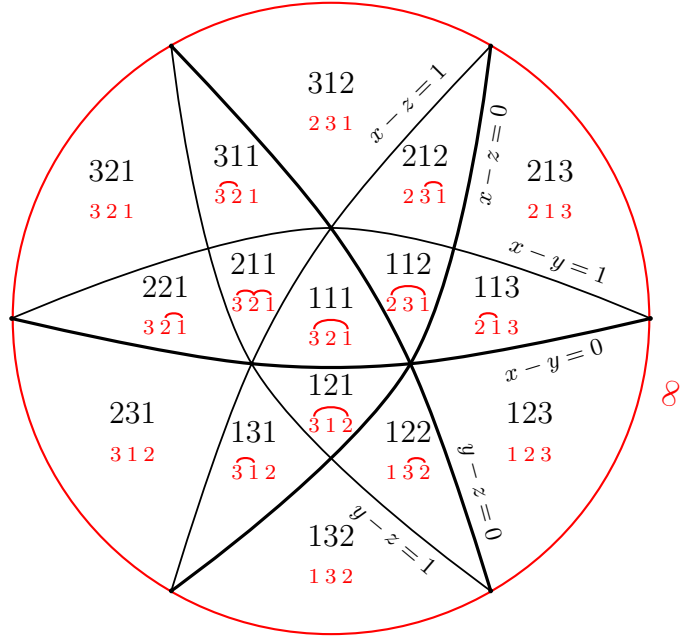

\section{Injectivity of $\lambda$}\label{inj}

\noindent
The proof of the injectivity of $\lambda$ is based on the following
lemma, where a particular case is considered. Beforehand, we introduce
a new concept.

\begin{definition}
  Let $w\in\WA$ for a subset $A$ of $[n]$, consider the poset of
  \emph{inversions of $w$}, $\inv(w):=\{(i,j)\mid i<j, w_i>w_j\}$,
  ordered so that $(i,j)\leq (k,\ell)\ \text{ if and only if }\
  [i,j]\subseteq [k,\ell]$.  Then, define $\maxinv(w)$ as the set of
  maximal elements of $\inv(w)$.
\end{definition}

\begin{lemma}\label{lema1}
  Let $A\subseteq[n]$, $v,w\in\WA$, and suppose that $P=(v,\fI)$ is a
  valid pair. If
$$\lambda(v,\fI)=\ob{w}\,,$$
then $v=w$ and $\fI=\maxinv(v)$.
\end{lemma}
\begin{proof}
  We first prove that $v=w$. Let $A=\{a_1,\dotsc,a_m\}$ with
  $a_1<\dotsb<a_m$, and suppose that, for $\pi,\rho\in\fS_m$,
  $v=a_{\pi_1}a_{\pi_2}\dotsb a_{\pi_m}$ and
  $w=a_{\rho_1}a_{\rho_2}\dotsb a_{\rho_m}$, and that, for some
  $1\leq\ell\leq n$, $\pi_i=\rho_i$ whenever $1\leq i<\ell$ but,
  contrary to our assumption, $\pi_\ell\neq\rho_\ell$. Finally, define
  $j,k>\ell$ such that $\rho_\ell=\pi_j$ and $\pi_\ell=\rho_k$ and
  $x:=a_{\pi_\ell}$, $y:=a_{\rho_\ell}$. Graphically, we have
\begin{align*}
  v&\;=\;w_1\;\dotsb\;w_{\ell-1}\;x{=}v_\ell\;v_
  {\ell+1}\,\dotsb\;y{=}v_j\;
  \dotsb\;v_m\\
  w&\;=\;w_1\;\dotsb\;w_{\ell-1}\;y{=}w_\ell\;w_
  {\ell+1}\;\dotsb\;x{=}w_k\;\dotsb\;w_m
\end{align*}
Then, for $a=o_P(y)<j$,
\begin{align*}
  &\ob{w}(y)=\ell-\Big|\big\{ 1\leq i<\ell\mid w_i>y\big\}\Big|\\
  &\hphantom{\ob{w}(y)}= j-\Big|\big\{ a\leq i<j\mid v_i>y\big\}\Big|
  \intertext{and hence} &j-\ell=\Big|\big\{\ell\leq i<j\mid
  v_i>y\big\}\Big|\boldsymbol{{}-{}}\Big|\big\{ 1\leq i<a\mid
  w_i>y\big\}\Big|\,.
  \intertext{This means that, for every $i$ with $\ell\leq i<j$,
    $w_i>y$ (and, in particular, $x>y$) and that, for every $i$ with
    $1\leq i<a$, $w_i\leq y$.  On the other hand, for
    $b=o_P(x)\leq\ell$,}
  &\ob{w}(x)=k-\Big|\big\{ 1\leq i<k\mid w_i>x\big\}\Big|\\
  &\hphantom{\ob{w}(x)}=\ell-\Big|\big\{ b\leq i<\ell\mid
  w_i>x\big\}\Big| \intertext{and} &k-\ell=\Big|\big\{\ell\leq i<k\mid
  w_i>x\big\}\Big|\boldsymbol{{}+{}}\Big|\big\{ 1\leq i<b\mid
  w_i>x\big\}\Big|
\end{align*}
Note that $b\leq a$ since $\ell<j$ and $P$ is a valid pair.  Then,
$\big\{ 1\leq i<b\mid w_i>x\big\}=\varnothing$ and $w_i>x$ for every
$i$ with $\ell\leq i<j$. In particular, $y>x$, which is absurd.  We
now leave it to the reader to prove that $\fI=\maxinv(v)$.
\end{proof}

\begin{corollary}\label{coro}
  Let $A\subseteq[n]$. The function $\ct\colon \WA\to\CFA\colon
  w\mapsto\ob{w}$ is a bijection.
\end{corollary}
\begin{proof}
  Since $\big|\WA\big|=\big|\CFA\big|=|A|!$, the result follows from
  the last lemma, since $\ct$ is injective.
\end{proof}

\begin{definition}$\ $
\begin{itemize}
\item We denote the inverse of $\ct$ by $\phiA\colon\CFA\to\WA$.
\item Given an $A$-parking function $f\colon A\to[n]$, the center of
  $f$, $Z(f)$, is the (unique\footnote{Note that if the restriction of
    $f$ to $X$ is $X$-central and the restriction of $f$ to $Y$ is
    $Y$-central for two subsets $X$ and $Y$ of $A$, then the
    restriction of $f$ to $(X\cup Y)$ is also $(X\cup Y)$-central.})
  maximal subset $Z$ of $A$ such that the restriction of $f$ to $Z$ is
  $Z$-central. Let $\zeta:=|Z|$ and note that $\zeta\neq0$ since
  $f^{-1}(1)\subseteq Z$ and $\big|f^{-1}(1)\big|\geq 1$.  Finally,
  let $f_Z\colon Z\to[n]$ be the restriction of $f$ to its center.
\end{itemize}
\end{definition}

\begin{lemma}\label{lema2}
  Let $f=\lambda(w,\fI)$ for a valid pair $P=(w,\fI)$, where $w\in\WA$
  for $A\subseteq[n]$ with $m=|A|$.
  \begin{enumerate}
  \item[\addtocounter{enumi}{1}
    \textbf{\arabic{section}.\arabic{theorem}.\arabic{enumi}.}]  Let,
    for some $p\geq 0$, $\fI=\big\{[o_1,c_1],\dotsc,[o_p,c_p]\big\}$
    with $o_1<\dotsb<o_p$.  Then,
$$f_Z=\ob{\w{1}{\zeta}}$$
and, in particular, $w([\zeta])=Z$.  Moreover, $\maxinv(\w{1}{\zeta})=
\{[o_1,c_1],\dotsc,[o_j,c_j]\}$ for some $0\leq j\leq p$.\\

\item[\addtocounter{enumi}{1}
  \textbf{\arabic{section}.\arabic{theorem}.\arabic{enumi}.}]  For
  every $j\in[m]$, $w_j\in Z(f)$ if and only if
$$f(w_j)=1+\Big|\big\{k<j\mid w_k<w_j\big\}\Big|\,.$$
\end{enumerate}
\end{lemma}
\begin{proof}$ $\\
  (\ref{lema2}.1) We start by proving the second statement, namely
  that $w([\zeta])=Z$.  Note that $w_1\in f^{-1}(\{1\})\subseteq Z$
  and suppose, contrary to our claim, that, for some $k<\zeta$ which
  we consider as small as possible, $w_k\notin Z$. Again, let $\ell>k$
  be as small as possible with $w_\ell\in Z$ and define $v=\w{1}{k}$.

  We now consider the ``restriction'' $w^*$ of $w$ to $Z$, that is,
  the subword of $w$ obtained by deleting all the elements of
  $[n]\setminus Z$, and let
  $$w':=\phiZ(f_Z)\in\WZ\,.$$
  By Lemma~\ref{lema1}, $w^*=w'$ and $k-f(w_\ell)$ is the number of
  integers greater than $w_\ell$ that precede it in $w^*$.  This means
  that $w_k,\dotsc,w_{\ell-1}>w_\ell$ and that $o(w_\ell)\leq k$.
  Hence, $k-f(w_k)$ is also the number of integers greater than $w_k$
  that precede it in $w$, and so $\ob{v}$ is the restriction of $f$ to
  $w([k])$, and $a\in Z$, a contradiction.
  Now, the result follows also from Lemma~\ref{lema1}.\\
  \ref{lema2}.2 is a clear consequence of \ref{lema2}.1. \end{proof}

We have proven that the ``initial parts'' of both $w$ and $\fI$ are
characterized by $f$. Let $m=|A|$, consider $c\in\N$ such that
$1<c\leq\zeta$, and define $\tw:=\w{c}{m}$; define also
$\tfI:=\varnothing$ if $j=p$, for $j,p$ defined as in the
statement of Lemma~\ref{lema2}, and
$\tfI:=\{\tI_1,\dotsc,\tI_{p-j}\}$, where
$$\tI_1=:[1,c_{j+1}-c+1],\dotsc, \tI_{p-j}:=[o_p-c+1,c_p-c+1]\,,$$
if $p>j$.  Suppose that, for some such $c$, $f$ also determines
$\tf:=\lambda(\tw,\tfI)$.  This proves our promised result (by
induction on $|A|$) and shows how to proceed for actually finding
$w\in\fS_n$ and $\fI$, given $f=\lambda(w,\fI)$: we find the center
$Z$ of $f$, build $\phiZ(f_Z)\in\WZ$ and $\tf$, find the center
$\tilde{Z}$ of $\tf$, build $\phitZ(f_{\tilde{Z}})\in\WtZ$ and
$\tilde{\tf}$, etc.

\newcounter{aaa}
\begin{definition}\label{defn}
  Given a parking function $f\in\PFA$, $f=\lambda(w,\fI)$, $m:=|A|$,
  $Z:=Z(f)$, and $\zeta:=|Z|<m$,
\begin{itemize}
\item let $b:=\min f(A\setminus Z)$ and
  $a:={\bf\max}\big(f^{-1}(\{b\})\setminus Z\big)$;
\item
  if $b>\zeta$, let $c:=b$;\\
  if $b\leq\zeta$, let $c$ be the greatest integer $i\in[\zeta]$ for
  which\\[-5pt]
\addtocounter{equation}{1}\addtocounter{aaa}{\arabic{equation}}
\item[]\strut\rlap{(\arabic{section}.\theaaa)}
\hfill$i+\big|w([i,\zeta])\cap [a-1]\big| = b\,.$\hfill\strut\\[-5pt]
\item let $X:= w([c-1])$ ($X\subseteq Z$ by
  Lemma~\ref{lema2});\\[-.5\baselineskip]
\item let$\strut$\\[-1.1\baselineskip]
  \hphantom{let$\strut$}$\begin{array}{rccl}
    \tf\colon&A\setminus X&\to&[m-c+1]\\[2pt]
    &x&\mapsto&
    \begin{cases}
      f(x)-\big|X\cap[x-1]\big|\,,& \text{if $x\in Z$\,;}\\
      f(x)-c+1\,,&\text{otherwise\,.}
\end{cases}
\end{array}$
\end{itemize}
\end{definition}

\begin{lemma}\label{lema3}
  With the definitions above,
\begin{enumerate}
\item[\addtocounter{enumi}{1}
  \textbf{\arabic{section}.\arabic{theorem}.\arabic{enumi}.}]
  $a=w_{\zeta+1}$ and $\strut a\in Z(\tf)$;
\item[\addtocounter{enumi}{1}
  \textbf{\arabic{section}.\arabic{theorem}.\arabic{enumi}.}]  $\strut
  Z\setminus X\subseteq Z(\tf)$;
\item[\addtocounter{enumi}{1}
  \textbf{\arabic{section}.\arabic{theorem}.\arabic{enumi}.}]  $\strut
  c=o_{(\tw,\tfI)}(a)$ and
\item[\addtocounter{enumi}{1}
  \textbf{\arabic{section}.\arabic{theorem}.\arabic{enumi}.}]  $\strut
  \tf=\lambda\big(\tw,\tfI\big)$.
\end{enumerate}
\end{lemma}
\begin{proof}
  If $b>\zeta$, then $X=Z$ and all the statements follow directly from
  the definitions. Hence, we consider that $b\leq\zeta$.  We start by
  seeing that $c$ is well defined.  Define $h\colon[\zeta]\to\N$ by
$$h(i)=i+\big|w([i,\zeta])\cap [a-1]\big|\,.$$
Then, for every $i<\zeta$, since $w([i,\zeta])=\{w_i\}\cup
w([i+1,\zeta])$, $h(i+1)$ either equals $h_i$ or $h_i+1$, depending on
whether $w_i$ is either less than $a$ or greater than $a$.  Since
$h(\zeta)\geq\zeta\geq b$, by definition, all we have to prove is that
$h(1)<b$, or, equivalently, that $1+\big|Z\cap[a-1]\big|<f_a$. But
$f_a\leq 1+\big|Z\cap[a-1]\big|$ implies that the restriction of $f$
to $Z':=Z\cup\{a\}$ is $Z'$-central, by Lemma~\ref{lema2}.2, which,
since $a\notin Z$, contradicts the maximality of $Z$. Note that the
set of values of $i$ for which (\arabic{section}.\theaaa) holds true
is an interval, and that its maximum, $c$, is the only one that is
greater than $a$.  By definition of $a$ and by Lemma~\ref{lema2}.1,
$a=w_{\zeta+1}$, for if $x=w_k$ and $a=w_\ell$ with $\ell>k$ and
$x>a$, then $f(x)\leq b$, by \eqref{comp}, and $x\in Z(f)$ .

Now, let $g=\lambda\big(\tw,\tfI\big)$ for $\tw$ and $\tfI$ as defined
before.  If $x\in A\setminus Z$, by definition of $\lambda$, viz.
\eqref{pdef}, $g(x)=f(x)-c+1=\tf(x)$.
In particular, $g(a)=1+\big|\tw([\zeta-c+1])\cap[a-1]\big|$.  Hence,
by Lemma~\ref{lema2}.2, $a\in Z(g)$. Now, Lemma~\ref{lema2}.1 implies
that $Z\setminus X$, the set of elements on the left side of $a$ in
$\tw$, is a subset of $Z(g)$, and that $c=o_{(\tw,\tfI)}(a)$. Now, the
last result, viz.  $g=\tf$, follows immediately, since for $x=w_j$
with $c\leq j\leq\zeta$, $f(x)=1+\big|w([j])\cap[x-1]\big|$ and
$g(x)=1+\big|\tw([j-c+1])\cap[x-1]\big|$.
\end{proof}
This concludes the proof of our main result.
\begin{proposition}
  The Pak-Stanley labeling is injective.\qed
\end{proposition}

\section{Inverse}
\noindent
It is easy to directly prove Corollary~\ref{coro} and even to
explicitly define $\phiA$, the inverse of $\ct$.  Nevertheless, we
consider here a method that we find very convenient, and particularly
well-suited to our purpose, the s-parking. Note that a similar method
is given by the depth-first search version of Dhar's burning algorithm
defined by Perkinson, Yang and Yu \cite{PYY}. In fact, it may be
proved that $Z(f)$ is the set of $\zeta$ visited vertices before the
first back-tracking, and that $\w{1}{\zeta}$ is given by the order in which
the vertices are visited.

\begin{definition}
  Let again $A=:\{a_1,\dotsc,a_m\}$ with $a_1<\dotsb<a_m$ and
  $f:A\to[m]$. For every $i\in[m]$, define the set
  $A_i:=\{a_1,\dotsc,a_i\}$, and define recursively the bijection
  $w^i\colon A_i \to[i]$ as follows.
 \begin{itemize}
 \item $w^1:a_1\mapsto 1$ (necessarily);
 \item for $1<j\leq i\leq m$,
   \begin{itemize}
   \item if $j<i$, $w^i(a_j)=\begin{cases}
       w^{i-1}(a_j), &\text{if $w^{i-1}(a_j)<f(a_i)$}\\
       1+w^{i-1}(a_j),&\text{if $w^{i-1}(a_j)\geq f(a_i)$}\end{cases}$
   \item $w^i(a_i)=f(a_i)$;
   \end{itemize}
 \end{itemize}
 Finally, let $\psi:[m]\to A$ be the inverse of $w^m:A\to[m]$.  We
 call $S(f):=\psi$ (viewed as the word {\footnotesize $\psi(1)\dotsb
   \psi(m)$}) the \emph{s-parking of $f$}.
 \end{definition}

 This operation resembles placing books on a bookshelf, where in step
 $i$ we want to put book $a_i$ at position $f(a_i)$ --- and so we must
 shift right every book already placed in a position greater than or
 equal to $f(a_i)$.  For example, if $A=\{3,4,6,7,8,9\}\subseteq[9]$
 and
 $f=\overset{\scriptscriptstyle3}{1}\overset{\scriptscriptstyle4}{1}
 \overset{\scriptscriptstyle6}{3}\overset{\scriptscriptstyle7}{4}
 \overset{\scriptscriptstyle8}{1}\overset{\scriptscriptstyle9}{4}$,
 then $S(f)=843967$. On the other hand, if $B=\{1,2,3,6,7,9\}$ and
 $g=\overset{\scriptscriptstyle1}{1}\overset{\scriptscriptstyle2}{2}
 \overset{\scriptscriptstyle3}{1}\overset{\scriptscriptstyle6}{2}
 \overset{\scriptscriptstyle7}{3}\overset{\scriptscriptstyle9}{2}$,
 then $S(g)=396712$. Finally, let $C=\{1,2,5,7\}$ and
 $h=\overset{\scriptscriptstyle1}{1}\overset{\scriptscriptstyle2}{2}
 \overset{\scriptscriptstyle5}{3}\overset{\scriptscriptstyle7}{1}$, so
 that $S(h)=7125$. The three constructions are used in the next
 example. See Figure~\ref{fig2}, where a parking function $f$ is
 represented on the top rows by orderly stacking in column $i$ the
 elements of $f^{-1}(i)$ (cf. \cite{Gar}), and row $j$ below the
 horizontal line is the inverse of $w^j$.  Note that \eqref{central}
 implies that $w^i$ is indeed a bijection for $i=1,\dotsc,m$.

 \begin{figure}[ht]
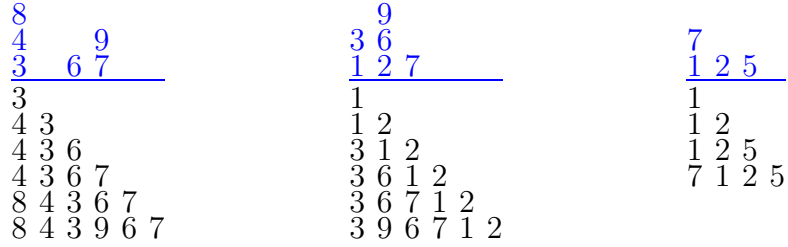

 {\blue\Large
 $$\begin{smallmatrix}
 8\\4&&&9\\3&&6&7\\[1.5pt]\hline\\[.5pt]
 {\black3}\\
 {\black4}&{\black3}\\
 {\black4}&{\black3}&{\black6}\\
 {\black4}&{\black3}&{\black6}&{\black7}\\
 {\black8}&{\black4}&{\black3}&{\black6}&{\black7}\\
 {\black8}&{\black4}&{\black3}&{\black9}&{\black6}&{\black7}
 \end{smallmatrix}\qquad\qquad
 \begin{smallmatrix}
 &9\\3&6\\1&2&7\\[1.5pt]\hline\\[.5pt]
 {\black1}\\
 {\black1}&{\black2}\\
 {\black3}&{\black1}&{\black2}\\
 {\black3}&{\black6}&{\black1}&{\black2}\\
 {\black3}&{\black6}&{\black7}&{\black1}&{\black2}\\
 {\black3}&{\black9}&{\black6}&{\black7}&{\black1}&{\black2}
 \end{smallmatrix}\qquad\qquad
\begin{smallmatrix} \vphantom{8}\\7\\1&2&5\\[1.5pt]\hline\\[.5pt]
 {\black1}\\
 {\black1}&{\black2}\\
 {\black1}&{\black2}&{\black5}\\
 {\black7}&{\black1}&{\black2}&{\black5}\\
 \vphantom{8}\\
 \vphantom{8}
 \end{smallmatrix}
$$
 }
 \caption{S-parking}
 \label{fig2}
 \end{figure}

\begin{lemma}
  Given $A$ and $f$ as in the previous definition,
  $f=\ob{S(f)}$. Conversely, given $A$ and $w\in\WA$,
  $w=S\big(\ob{w}\big)$.
\end{lemma}
\begin{proof}
  Let $w=S(f)$ and $\psi=w^{-1}$ and note that, when we s-park $f$,
  each element $a_i$ of $A$ is put first at position $f(a_i)$, and it
  is shifted one position to the right by an element $a_j$ if and only
  if $j>i$ and $\pi_j<\pi_i$; it ends at position $\psi_i$. Hence,
  $f=\ob{S(f)}=\ob{w}$. Then $S$ is the inverse of $\ct$, that is,
  $S=\phiA$.
  \end{proof}

\begin{examplef}
  Let us recover the valid pair $P=\lambda^{-1}(f)$ out of
  $f=341183414$. In the first column, on the right, the elements of
  the center of $f$ are written in italic and $a$ is written in
  boldface.  The last column may be obtained by s-parking, as
  represented in Figure~\ref{fig2}.

$$\begin{array}{|cl|c|c|c|ccl|}
\hline
f&&a&b&c&f_Z&& \\
\hline
\overset{\scriptscriptstyle1}{3}\overset{\scriptscriptstyle2}{4}
\overset{\scriptscriptstyle3}{1}\overset{\scriptscriptstyle4}{1}
\overset{\scriptscriptstyle5}{8}\overset{\scriptscriptstyle6}{3}
\overset{\scriptscriptstyle7}{4}\overset{\scriptscriptstyle8}{1}\overset{\scriptscriptstyle9}{4}&
{\begin{smallmatrix}
\mathit{8}&&&\mathit{9}&&\\
\mathit{4}&&\mathit{6}&\mathit{7}&&\\[-4.5pt]
\mathit{3}&\tvs&\mathbf{1}&2&\tvs&\tvs&\tvs&5&\tvs\\
\end{smallmatrix}}&1&3&3&
\overset{\scriptscriptstyle3}{1}\overset{\scriptscriptstyle4}{1}
\overset{\scriptscriptstyle6}{3}\overset{\scriptscriptstyle7}{4}
\overset{\scriptscriptstyle8}{1}\overset{\scriptscriptstyle9}{4}&
= & \ob{843967}
\nccurve[nodesep=.05,ncurv=0.35,angleB=90,angleA=90]{-}{x1}{x2}\\
\hline
\overset{\scriptscriptstyle1}{1}\overset{\scriptscriptstyle2}{2}
\overset{\scriptscriptstyle3}{1}\overset{\scriptscriptstyle5}{6}
\overset{\scriptscriptstyle6}{2}\overset{\scriptscriptstyle7}{3}
\overset{\scriptscriptstyle9}{2}&
{\begin{smallmatrix}
    &\mathit{9}&\\
    \mathit{3}&\mathit{6}&\\[-4.5pt]
    \mathit{1}&\mathit{2}&\mathit{7}&\tvs&\tvs&\mathbf{5}&\tvs
\end{smallmatrix}}&5&6&4&
\overset{{\scriptscriptstyle1}}{1}\overset{{\scriptscriptstyle2}}{2}
\overset{{\scriptscriptstyle3}}{1}\overset{{\scriptscriptstyle6}}{2}
\overset{{\scriptscriptstyle7}}{3}\overset{{\scriptscriptstyle9}}{2}&
= & \ob{396712}
\nccurve[nodesep=.05,ncurv=0.35,angleB=90,angleA=90]{-}{y1}{y2}\\
\hline
\overset{{\scriptscriptstyle1}}{1}\overset{{\scriptscriptstyle2}}{2}
\overset{{\scriptscriptstyle5}}{3}\overset{{\scriptscriptstyle7}}{1}&
{\begin{smallmatrix}
\mathit{7}\\
\mathit{1}&\mathit{2}&\mathit{5}
\end{smallmatrix}}&{}-{}&{}-{}&{}-{}&
\overset{{\scriptscriptstyle1}}{1}\overset{{\scriptscriptstyle2}}{2}
\overset{{\scriptscriptstyle5}}{3}\overset{{\scriptscriptstyle7}}{1}&
= & \ob{7125}
\nccurve[nodesep=.05,ncurv=0.35,angleB=90,angleA=90]{-}{z1}{z2}\\[4.5pt]
\hline
\end{array}$$
\end{examplef}

In fact, as we know,
$f=\rnode{c8}{8}4\rnode{c3}{3}96\rnode{c7}{7}1\rnode{c2}{2}\rnode{c5}{5}
\nccurve[nodesep=.05,ncurv=0.35,angleB=90,angleA=90]{-}{c8}{c7}
\nccurve[nodesep=.05,ncurv=0.35,angleB=90,angleA=90]{-}{c3}{c2}
\nccurve[nodesep=.05,ncurv=0.35,angleB=90,angleA=90]{-}{c7}{c5}$, that
is $P=\big(843967125,\{[1,6],[3,8],[6,9] \}\big)$.


\begin{thebibliography}{9}

\bibitem{DGO} R.~Duarte and A.~Guedes~de~Oliveira, The braid and the
  Shi arrangements and the Pak-Stanley labeling,
  \textit{Eur. J. Combinatorics}, in press.
\bibitem{Gar} A.~M.~Garsia and M.~Haiman, A Remarkable $q,t$-Catalan
  Sequence and $q$-Lagrange Inversion, \textit{J.\ Algebr.\ Comb.}
  \textbf{5} (1996) 191--244.
\bibitem{KW} A.~Konheim and B.~Weiss, An occupancy discipline and
  applications, \textit{SIAM J.\ Appl.\ Math.}  \textbf{14} (1966),
  1266--1274.
\bibitem{PYY}D.~Perkinson, Q.~Yang and K.Yu, $G$-parking functions and
  tree inversions, \textit{arXiv:1309.2201 [math.CO]}, in press.

\bibitem{Shi} J.~Y.~Shi, \textit{The Kazhdan-Lusztig Cells in certain
    Affine Weyl Groups}, Lecture Notes in Mathematics \textbf{1179}
  (1986), Springer-Verlag .
\bibitem{stan1} R.~Stanley, An introduction to hyperplane
  arrangements, in \textit{Geometric Combinatorics} (E.~Miller,
  V.~Reiner, and B.~Sturmfels, eds.), IAS/Park City Mathematics
  Series, vol.\textbf{13}, A.M.S.\ (2007), 389--496.
\bibitem{stan2} R.~Stanley, Hyperplane arrangements, interval orders
  and trees, \textit{Proc.\ Nat.\ Acad.\ Sci.} \textbf{93} (1996),
  2620--2625.
\end{thebibliography}
\end{document}